\documentclass[12pt]{article}
\usepackage{amsmath}
\usepackage{amsfonts}
\usepackage{amssymb}
\usepackage{amsthm}
\usepackage{mathrsfs}
\usepackage{indentfirst}
\usepackage{bbm}
\usepackage{bm}
\newtheorem{theorem}{\sc{\bf Theorem}}
\newtheorem{theorem A}{\sc{\bf Theorem A}}
\newtheorem{lemma}{\sc{\bf Lemma}}

\begin{document}
\title{$L^p$ regularity of the Bergman projection on generalizations of the
	Hartogs triangle in $\mathbb{C}^{n+1}$}
\author{Qian Fu\thanks{ School of Mathematical Sciences, Beijing Normal University, fqian19@126.com},~ Guan-Tie Deng\thanks{Corresponding author. School of Mathematical Sciences, Beijing Normal University, 96022@bnu.edu.cn}, Hui Cao\thanks{ School of Mathematical Sciences, Beijing Normal University, 18166038936@163.com}}

\date{}
\maketitle
\begin{center}
\begin{minipage}{120mm}
\begin{center}{\bf Abstract}\end{center}
{In this paper we investigate a class of domains $\Omega^{n+1}_k =\{(z,w)\in \mathbb{C}^n\times \mathbb{C}: |z|^k < |w| < 1\}$ for $k \in \mathbb{Z}^+$  which generalizes the Hartogs triangle. we first obtain the new explicit formulas for the Bergman kernel function on these domains and further give a range of $p$ values for which the $L^p$ boundedness of the Bergman projection holds. This range of $p$ is shown to be sharp. }\\


{\bf Data availability statements}
The data that support the findings of this study are available from the corresponding author upon reasonable request.

\end{minipage}
\end{center}

\section{Introduction}
 \setlength{\parindent}{2em} Let $\Omega$ be a domain in the $n$-dimensional complex space $\mathbb{C}^{n}$. For $p>0$, denote
$$
L^{p}(\Omega)=\left\{f:\left(\int_{\Omega}|f|^{p} d V\right)^{\frac{1}{p}}:=\|f\|_{p}<\infty\right\},
$$
where $dV(z)$ is the ordinary Lebesgue volume measure on $\Omega$. As we known, $L^{p}(\Omega)$ is a Banach space when $p>1$. For $p=2$, $L^{2}(\Omega)$ is a Hilbert space with the inner product:
\begin{equation}\label{innerproduct}
\langle f,g\rangle=\int_{\Omega}f(z)\overline{g(z)}dV(z).
\end{equation}
Let $\mathcal{O}(\Omega)$ denote the holomorphic functions on $\Omega$, and $A^{p}(\Omega)=\mathcal{O}(\Omega) \cap L^{p}(\Omega)$.
The Bergman projection associated to $\Omega$ will be written $\mathbf{P}_{\Omega}$, or $\mathbf{P}$ if $\Omega$ is clear, and is the orthogonal projection operator $\mathbf{P}: L^{2}(\Omega) \longrightarrow A^{2}(\Omega)$. It is elementary that $\mathbf{P}$ is self-adjoint with respect to the inner product (\ref{innerproduct}). The Bergman kernel, denoted $B_{\Omega}(z, w)$,  satisfies
$$
\mathbf{P}_{\Omega} f(z)=\int_{\Omega} B_{\Omega}(z, w) f(w) d V(w), \quad f \in L^{2}(\Omega).
$$
Given an orthonormal Hilbert space basis $\left\{\phi_{\alpha}\right\}_{\alpha \in \mathcal{A}}$ for $A^{2}(\Omega)$, the Bergman kernel is given by the following formula:
$$
B_{\Omega}(z, w)=\sum_{\alpha \in \mathcal{A}} \phi_{\alpha}(z) \overline{\phi_{\alpha}(w)} .
$$

We use the following notation to simplify writing various inequalities. If $A$ and
$B$ are functions depending on several variables, write $A\lesssim B$ to signify that there exists a
constant $K > 0$, independent of relevant variables, such that $A \leqslant K B$. The independence
of which variables will be clear in context. Also write $A\thickapprox B$ to mean that $A\lesssim B \lesssim A$.

Chakrabarti et al. study the Bergman kernel of an
elementary Reinhardt domain $\mathcal{H}(k)$ associated to a multi-index
$k=(k_1,\dots,k_n)\in \mathbb{Z}^n$ defined as
$$ \mathcal{H}(k)=\{z\in \mathbb{D}^n:z^k\text{ is defined and }|z^k|<1\},$$
where $\mathbb{D}^n$, denotes the unit polydisc in $\mathbb{C}^n$.

In this article, we mainly study the following bounded regions. For $k \in \mathbb{Z}^+$, we define the domain $\Omega^{n+1}_k\subseteq \mathbb{C}^n\times \mathbb{C}$ by
$$\Omega^{n+1}_k =\{(z,w)\in \mathbb{C}^n\times \mathbb{C}: |z|^k < |w| < 1\}$$
and call $\Omega^{n+1}_k$ the generalized Hartogs triangle of exponent $k,n$. Edholm  gives the Bergman kernel of 
$$
\mathbb{H}_{\gamma}=\{(z_1,z_2) \in \mathbb{C}^2 \colon |z_1|^{\gamma}<|z_2| < 1\}
$$
for certain values of $\gamma>0$, including all positive integers in \cite{LD1,LD3}, and further
Edholm, McNeal \cite{LD4} and Chen \cite{CL2} continue their study of the Bergman projection on the generalized Hartogs triangles and obtain Sobolev estimates.

Besides the power-generalized Hartogs triangles $\mathbb{H}_{\gamma}$ investigated by Chakrabarti and Zeytuncu \cite{ZB} , Edholm and McNeal \cite{LD2}, Beberok \cite{TB} also considered the $L^p$ boundedness of the Bergman projection on the following generalization of the Hartogs triangle
$$\mathcal{H}^{n+1}_k := \{(z, w) \in \mathbb{C}^n\times \mathbb{C}: ||z|| < |w|^k < 1\},$$
where $k \in \mathbb{Z}^+$ and $||\cdot||$ is the Euclidean norm in $\mathbb{C}^n$. It can be seen that this area is very similar to our research. On some other domains, the projection has only a finite range of mapping regularity  (see, for example, \cite{LD2,HB,CL1}, etc).

 Deng et al. obtained a general form of weighted Bergman reproducing kernel in \cite{DG2}, which can be calculated concrete Bergman kernel functions for specific weights and domains. The purpose of this paper is to explore the expression of Bergman kernel in a more general area $\Omega^{n+1}_k$ on the basis of Edholm's work. Meanwhile, we adopt a more concise calculation method, which can be realized by computer calculation and data processing.  Further, we investigate the regularity of Bergman projection operators on $\Omega^{n+1}_k$ in Theorem \ref{thm3} by the estimate of Bergman kernel which obtained from Theorem \ref{thm1}.

\section{Main results}
\label{result}

The precise statement of our main result is as follows.
\begin{theorem}
\label{thm1}
Let $a=w\cdot\overline{t}, b=z\cdot\overline{s}.$ The Bergman kernel for the generalized Hartogs triangle $\Omega^{n+1}_k$ is given by
$$
B_{k,n}((z, w),(s,t))
=n!\frac{g_{(n+1) k}(b) a^{n+1}+\cdots+g_{k}(b) a+g_{0}(b)}{\pi^{n+1} k^{2}\left(1-a\right)^{2}\left(a-b^{k}\right)^{n+1}},
$$
where $g_{lk}(b) (0\leq l\leq n+1)$ is the polynomial
$$
g_{lk}(b)
=\sum_{\substack{j_1+j_2+\cdots+j_{n+3}=(n-l+2)k-(n+1),\\ 0\leq j_i\leq k-1}} b^{j_3+j_4+\cdots +j_{n+3}}.
$$
\end{theorem}

Next, to better understand the calculation process,  we give the special case of the Bergman kernel for the generalized Hartogs triangle $\Omega^{2+1}_k$ when $n = 2$.

\begin{theorem}
\label{thm2}
When $n=2$, let $a=w\cdot\overline{t}$, $b=z\cdot\overline{s}.$ The Bergman kernel for the generalized Hartogs triangle $\Omega^{2+1}_k$ is given by
$$
B_{k,2}((z, w),(s,t))
=\frac{g_{3k}(b) a^{3}+g_{2k}(b) a^{2}+g_{k}(b) a+g_{0}(b)}{\pi^{3} k^{2}\left(1-a\right)^{2}\left(a-b^{k}\right)^{3}},
$$
where $g_{lk}(b)$ $(0\leq l\leq 3)$ is the polynomial as follows

$g_0(b)=\sum\limits_{m=0}^{k-2}(m+1)(k-m)(k-m-1) b^{2k-1+m},$

$g_{k}(b)=\sum\limits_{m=0}^{k-1}((k-m-1)^2(k-m)b^{2k-1}+(m+1)((k+m+1)(k+m)-3(m+1)m)b^{k-1}) b^{m},$

$g_{2k}(b)=\sum\limits_{m=0}^{k-2}((k-m-1)((k+m+1)(k+m)-3(m+1)m)b^{k-1}+(m+2)^2(m-1)) b^{m},$

$g_{3k}(b)=\sum\limits_{m=0}^{k-3}(k-2-m)(m+2)(m+1) b^{m}.$
\end{theorem}

Next, we introduce $L^p$ regularity results of Bergman projection $\mathbf{P}$ on $\Omega^{n+1}_k$.

\begin{theorem}
\label{thm3}
	The Bergman projection $\mathbf{P}$ is a bounded operator from  $L^{p}\left(\Omega^{n+1}_k\right)$ to $L^{p}\left(\Omega^{n+1}_k\right)$
	if and only if $p \in\left(\frac{2 k+2n}{k+n+1}, \frac{2 k+2n}{k+n-1}\right) .$
\end{theorem}

\section{Proof of the theorems}
Before giving the main result of this work we prove the following lemmas that will
be used to prove the main theorem.

\begin{lemma}{\rm\cite{HB}}
\label{formula}
For any $v_1,\cdots, v_n \geq 0$,
$$
\int_{S^{2n-1}}|\zeta_1|^{2v_1}\cdots|\zeta_n|^{2v_n}d\sigma(\zeta)=\frac{2v!\pi^n}{\Gamma(n+|v|)},
$$
where $|v| = v_1 +\cdots+ v_n$, $v! = \Gamma(v_1 + 1)\cdots \Gamma(v_n + 1),$
and $S^{2n-1}$ is unit sphere in $\mathbb{C}^n$ with respect to the surface measure $d\sigma$.
\end{lemma}

Let $B_{k,n}$ denote the Bergman kernel on $\Omega^{n+1}_k$ and let's calculate $B_{1,n}$ first. In addition to the method of the product of kernel functions, here we give another solution to calculate Bergman kernel.
\begin{lemma}
The Bergman kernel for the generalized Hartogs triangle $\Omega^{n+1}_1$ is given by
$$
B_{1,n}((z,w),(s,t))=\frac{n!w\cdot\overline{t}}{\pi^{n+1}(1-w\cdot\overline{t})^2(w\cdot\overline{t}-z\cdot\overline{s})^{n+1}}.
$$
\end{lemma}
\begin{proof}
Firstly, we consider the Bergman kernel $B_{k,n}((z,w),(s,t))$, $z,s \in\mathbb{C}^n, w,t \in\mathbb{C}$. Choosing one unitary matrices $U$  such that $z=|z|\mathbbm{1}U^{-1},$ where $\mathbbm{1}=(1,0,\cdots,0)$.  Then we get
\begin{align}\label{Bk}
B_{k,n}((z,w),(s,t))
&=\sum\limits_{p\in\mathbb{N}^n,q\in\mathbb{Z}}\frac{z^{p}w^{q}(\overline{s^pt^q})}{\gamma_{p,q}}
 =\sum\limits_{p\in\mathbb{N}^n,q\in\mathbb{Z}}\frac{(|z|\mathbbm{1})^{p}w^{q}(\overline{(sU)^pt^q})}{\gamma_{p,q}}\notag\\
&=\sum\limits_{p_1,q\in \Lambda_{k,n}}\frac{(|z|\mathbbm{1}\overline{U}^{T}\overline{s}^{T})^{p_1}(w\cdot \overline{t})^{q}}{N_{k,n}(p_1,q)}
 =\sum\limits_{p_1,q\in \Lambda_{k,n}}\frac{(z\cdot\overline{s})^{p_1}(w\cdot\overline{t})^{q}}{N_{k,n}(p_1,q)},
\end{align}
where $p=(p_1,p_2,\cdots,p_n)\in\mathbb{N}^n$, $\gamma_{p,q}=\int_{\Omega^{n+1}_k}|z^p|^2|w^q|^2dV,$ $\Lambda_{k,n}=\{(p_1,q):p_1+k(q+1)>-n\},$ $N_{k,n}(p_1,q)=\int_{\Omega^{n+1}_k}|z_1|^{2p_1}|w|^{2q}dV$.

Thus, in order to compute $B_{1,n}((z,w),(s,t))$, we need to calculate the function $N_{1,n}(p_1,q)$. From Lemma \ref{formula}, we get
\begin{align*}
N_{1,n}(p_1,q)=\int_{\Omega^{n+1}_1}|z_1|^{2p_1}|w|^{2q}dV
&= \int_{0<|w|<1}|w|^{2q}\int_{0}^{|w|}\int_{|\xi|=1}|\xi_1|^{2p_1}r^{2p_1+2n-1}dV(\xi)drdV(w)\\
&= \frac{\pi^n \Gamma(p_1+1)}{(p_1+n)\Gamma(p_1+n)}\int_{0<|w|<1}|w|^{2q+2p_1+2n}dV(w)\\
&= \frac{\pi^{n+1}\Gamma(p_1+1)}{ (p_1+n)(p_1+n+q+1)\Gamma(p_1+n)}.
\end{align*}

Let $a=z\cdot\overline{s}, b=w\cdot\overline{t}$. When $k= 1,$ we have
$|z|< |w| < 1$, $|s|<|t| < 1$, thus $|b|<1$,
\begin{align*}
B_{1,n}((z,w),(s,t))
&=\sum\limits_{p_1,q\in \Lambda_{1,n}}\frac{(z\cdot\overline{s})^{p_1}(w\cdot\overline{t})^{q}}{N_1(p_1,q)}\\
&= \sum\limits_{p_1,q\in \Lambda_{1,n}}\frac{(p_1+n)(p_1+n+q+1)\Gamma(p_1+n)}{\pi^{n+1}\Gamma(p_1+1)}a^{p_1}b^{q}\\
&=\frac{1}{\pi^{n+1}}\sum\limits_{p_1=0}^\infty\frac{(p_1+n)\Gamma(p_1+n)a^{p_1}}{\Gamma(p_1+1)}\cdot\left(\sum\limits_{q=-n-p_1}^\infty(p_1+n+q+1)b^q\right).
\end{align*}
Notice that
\begin{align*}
\sum\limits_{q=-n-p_1}^\infty(p_1+n+q+1)b^q
&=\sum\limits_{q=0}^\infty(q+1)b^{q-n-p_1}\\
&=\frac{1}{b^{n+p_1}}\sum\limits_{q=0}^\infty(q+1)b^{q}
=\frac{1}{b^{n+p_1}}\frac{d}{db}\left(\sum\limits_{q=0}^\infty b^{q+1}\right)\\
&=\frac{1}{b^{n+p_1}}\frac{d}{db}\left(\frac{b}{1-b}\right)=\frac{1}{b^{n+p_1}(1-b)^2}.
\end{align*}

Let $\zeta=\frac{a}{b}=\frac{z\cdot\overline{s}}{w\cdot\overline{t}}$, then $|\zeta|=\frac{|z\cdot\overline{s}|}{|w||\overline{t}|}\leq\frac{|z||\overline{s}|}{|w||\overline{t}|}<1$,
\begin{align*}
B_{1,n}((z,w),(s,t))
&=\frac{1}{\pi^{n+1}b^{n}(1-b)^2}\sum\limits_{p_1=0}^\infty(p_1+n)\cdots(p_1+1)\zeta^{p_1}\\
&=\frac{1}{\pi^{n+1}b^{n}(1-b)^2}\frac{d^n}{d\zeta^n}\sum\limits_{p_1=0}^\infty\zeta^{p_1+n}\\
&=\frac{n!}{\pi^{n+1}b^{n}(1-b)^2({1-\zeta})^{n+1}}=
\frac{n!w\cdot\overline{t}}{\pi^{n+1}(1-w\cdot\overline{t})^2(w\cdot\overline{t}-z\cdot\overline{s})^{n+1}}.
\end{align*}

\end{proof}

\begin{lemma}
(Bell's transformation rule {\rm\cite{BELL}}) Let $\Omega$ and $\bar{\Omega}$ be domains in $\mathbb{C}^{n}$
with respective Bergman kernels $\mathbb{B}$ and $\overline{\mathbb{B}}$, and suppose $\phi$ is a proper holomorphic
map of onder $k$ from $\Omega$ onto $\bar{\Omega}$. Let $u:=\operatorname{det}\left[\phi^{\prime}\right]$, and let $\Phi_{1}, \Phi_{2}, \cdots, \Phi_{k}$ be the branch inverses of $\phi$ defined locally on $\bar{\Omega}-V$, where $V:=\{\phi(z): u(z)=0\} .$ Finally, write
$U_{j}:=\operatorname{det}\left[\Phi_{j}^{\prime}\right] .$ Then,
$$
u(z) \overline{\mathbb{B}}(\phi(z), w)=\sum_{j=1}^{k} \mathbb{B}\left(z, \Phi_{j}(w)\right) \overline{U_{j}(w)} .
$$
\end{lemma}
We're now ready to compute the Bergman kernel of generalized Hartogs triangles by using the above lemmas.

\subsection{Proof of Theorem 1}
\begin{proof}

Firstly, we define the map $\phi$ and its local inverses $\Phi_{1}, \cdots, \Phi_{k}$. For each integer $k \geq 2$, $\phi: \Omega^{n+1}_{1} \rightarrow \Omega^{n+1}_{k}$ given by $\phi(z,w)=\left(z, w^{k}\right):=\left(\phi_{1}(z,w), \phi_{2}(z,w)\right)$ is a branched
cover of order $k$, since
$$
\begin{aligned}
\left\{\left|\phi_{1}(z,w)\right|^{k}<\left|\phi_{2}(z,w)\right|<1\right\} & \Longleftrightarrow\left\{\left|z\right|^{k}<\left|w^{k}\right|<1\right\} \\
& \Longleftrightarrow\left\{\left|z\right|<\left|w\right|<1\right\}
\end{aligned}
$$
We note $u(z,w)=\operatorname{det}\left[\phi^{\prime}\right]=k w^{k-1}$, so $V=\{\phi(z): u(z)=0\}$ is the set $\left\{w=0\right\}$, which is disjoint from $\Omega^{n+1}_{k}$. For each $j=1, \ldots, k$, the map $\Phi_{j}(z,w)=\left(z, \zeta^{j} w^{1 / k}\right)$ defines a local inverse of $\phi$, where $\zeta=e^{2 \pi i / k}$ and $w^{1 / k}$ is taken to mean the root with argument in the interval $\left(0, \frac{2 \pi}{k}\right)$. From this, we see $U_{j}(z,w)=\operatorname{det}\left[\Phi_{j}^{\prime}\right]=\left(\zeta^{j} w^{1 / k-1}\right) / k$. We now apply Bell's rule:
\begin{align}\label{proof1.1}
B_{k,n}((z, w^{k}),(s,t))
&=\frac{\bar{t}^{\frac{1}{k}-1}}{k^{2} w^{k-1}} \sum_{j=1}^{k} B_{1,n}(z,w),(s, \zeta^{j}t^{\frac 1k})) \bar{\zeta}^{j} \notag\\
&=\frac{\bar{t}^{\frac{1}{k}-1}}{k^{2} w^{k-1}} \sum_{j=1}^{k} \frac{n!w\cdot\overline{\zeta^{j}t^{\frac 1k}}\bar{\zeta}^{j}}{\pi^{n+1}(1-w\cdot\overline{\zeta^{j}t^{\frac 1k}})^2(w\cdot\overline{\zeta^{j}t^{\frac 1k}}-z\cdot\overline{s})^{n+1}} \notag\\
&=\frac{n!a^{2-k}}{\pi^{n+1} k^{2}} \sum_{j=1}^{k} \frac{\bar{\zeta}^{2 j}}{\left(1-a \bar{\zeta}^{j}\right)^{2}\left(a \bar{\zeta}^{j}-b\right)^{n+1}}\notag\\
&=\frac{n!a^{2-k}}{\pi^{n+1} k^{2}} \sum_{j=1}^{k} \frac{\zeta^{(n+1) j}}{\left(\zeta^{j}-a\right)^{2}\left(a-b\zeta^{j}\right)^{n+1}},
\end{align}
where $a=w \bar{t}^{\frac{1}{k}}$ and $b=z \bar{s}$. Define $f_{j}(a, b):=\left(\zeta^{j}-a\right)^{2}\left(a-b\zeta^{j}\right)^{n+1}$ and notice that
$$
\prod_{j=1}^{k} f_{j}(a, b)=\prod_{j=1}^{k}\left(\zeta^{j}-a\right)^{2} \cdot \prod_{j=1}^{k}\left(a-b \zeta^{j}\right)^{n+1}=\left(1-a^{k}\right)^{2}\left(a^{k}-b^{k}\right)^{n+1}.
$$
Now, it follows that
\begin{align}\label{proof1.2}
(\ref{proof1.1}) &=\frac{n!a^{2-k}}{\pi^{n+1} k^{2}} \sum_{j=1}^{k} \frac{\zeta^{(n+1) j}}{f_{j}(a, b)}\notag \\
&=\frac{n!a^{2-k} \sum_{j=1}^{k} F_{j}(a, b) \zeta^{(n+1) j}}{\pi^{n+1} k^{2}\left(1-a^{k}\right)^{2}\left(a^{k}-b^{k}\right)^{n+1}},
\end{align}
where $F_{j}(a, b):=\frac{\left(1-a^{k}\right)^{2}\left(a^{k}-b^{k}\right)^{n+1}}{f_{j}(a, b)}$. Notice that each $F_{j}(a, b)$ can be written as a polynomial in $a$ of degree $(n+3)k-(n+3)$, so the numerator of $(\ref{proof1.2})$ takes the following form:
$$
a^{2-k} \sum_{j=1}^{k} F_{j}(a, b) \zeta^{(n+1)j}=\sum_{j=2-k}^{(n+2)k-(n+1)} g_{j}(b) a^{j}:=G(a, b) .
$$
We now wish to calculate the coefficient polynomials $g_{j}(b)$. Toward this goal,
observe that $G\left(\zeta^{m} a, b\right)=G(a, b)$ for all $m \in \mathbb{Z}$. This follows because
$$
\begin{aligned}
G\left(\zeta^{m} a, b\right) &=\left(\zeta^{m} a\right)^{2-k} \sum_{j=1}^{k} F_{j}\left(\zeta^{m} a, b\right) \zeta^{(n+1) j} \\
&=a^{2-k} \sum_{j=1}^{k} \frac{\left(1-a^{k}\right)^{2}\left(a^{k}-b^{k}\right)^{n+1}}{f_{j-m}(a, b)} \zeta^{(n+1)(j-m)}=G(a, b).
\end{aligned}
$$
Here, we've used the facts that $f_{j}\left(\zeta^{m} a, b\right)=\zeta^{(n+3) m} f_{j-m}(a, b)$ and $f_{j}(a, b)=f_{j+m k}(a, b)$
for all $m \in \mathbb{Z}$. Because $G$ has this invariance, we conclude that
$$
G(a, b)=a^{2-k} \sum_{j=1}^{k} F_{j}(a,b) \zeta^{(n+1)j}=g_{(n+1) k}(b) a^{(n+1)k}+\cdots+g_{k}(b) a^{k}+g_{0}(b) .
$$

We can see by observing that $g_{l k}(b)$ corresponds to the coefficient
functions of the $a^{(l+1)k-2}$ terms of $F_{j}(a, b) \zeta^{(n+1)j}$, where $l$ is a non-negative integer and $0\leq l\leq n+1$. Now we're going to write $\theta=\zeta^j$ for convenience, $\theta^k=1$.
\begin{align*}
F_{j}(a, b)
&=\frac{\left(1-a^{k}\right)^{2}\left(a^{k}-b^{k}\right)^{n+1}}{f_{j}(a, b)}
 =(\frac{a^k-1}{a-\theta})^2(\frac{a^k-b^k}{a-b\theta})^{n+1}\\
&=a^{(n+3)k-(n+3)}(\sum_{j_1=0}^{k-1}(\frac{\theta}{a})^{j_1})^2(\sum_{j_3=0}^{k-1}(\frac{b\theta}{a})^{j_3})^{n+1}\\
&=a^{(n+3)k-(n+3)}\sum_{m=0}^{(n+3)k-(n+3)}\sum_{\substack{j_1+j_2+\cdots+j_{n+3}=m,\\ 0\leq j_i\leq k-1}} b^{j_3+j_4+\cdots +j_{n+3}}(\frac{\theta}{a})^m,
\end{align*}
then, for the coefficient of the $a^{(l+1)k-2}$ terms of $F_{j}(a, b) \zeta^{(n+1)j}$ $i.e.$
$m=(n-l+2)k-(n+1)$, we can get
\begin{align*}
g_{l k}(b)
&=\zeta^{(n+1)j}(\theta)^{(n-l+2)k-(n+1)}\sum_{\substack{j_1+j_2+\cdots+j_{n+3}=(n-l+2)k-(n+1),\\ 0\leq j_i\leq k-1}} b^{j_3+j_4+\cdots +j_{n+3}}\\
&=\sum_{\substack{j_1+j_2+\cdots+j_{n+3}=(n-l+2)k-(n+1),\\ 0\leq j_i\leq k-1}} b^{j_3+j_4+\cdots +j_{n+3}}.
\end{align*}
We have completed the proof of the theorem  by formally replacing the variable $w^k$ with $w$.
\end{proof}

\begin{lemma}\label{lemma3}
Let $l \in \mathbb{N}$, $k\in \mathbb{Z}^{+}$, $(j_1, j_2) \in \mathbb{N}^2$ satisfying the following conditions:

 $i)j_1+j_2=l$

 $ii)0\leq j_i\leq k-1  (i=1,2).$
 
$f(l)$ is the number of $(j_1, j_2)$, then
$$ f(l)=\left\{
\begin{array}{rcl}
l+1       &      & {0\leq l\leq k-1}\\
2k-l-1   &      & {k-1\leq l\leq 2k-2}\\
0     &      & {l>2k-2}
\end{array} \right. $$
\end{lemma}

\begin{lemma}\label{lemma4}
Let $l\in \mathbb{N}$, $k\in \mathbb{Z}^{+}$, $(j_3,j_4,j_5) \in \mathbb{N}^3$ satisfying the following conditions:

 $i)j_3+j_4+j_5=l$

 $ii)0\leq j_i\leq k-1  (i=3,4,5).$
 
 $h(l)$ is the number of $(j_1, j_2)$, then
$$ h(l)=\left\{
\begin{array}{rcl}
\frac{(l+2)(l+1)}{2}       &      & {0\leq l\leq k-1}\\
\frac{(l+2)(l+1)}{2}-\frac{3(l+2-k)(l+1-k)}{2}    &      & {k-1\leq l\leq 2k-2}\\
\frac{(3k-l-1)(3k-l-2)}{2}      &      & {2k-2\leq l\leq 3k-3}\\
0      &      & { l> 3k-3}
\end{array} \right. $$
\end{lemma}

Basing on the above two lemmas, we can complete the proof of Theorem 2.
\subsection{Proof of Theorem 2}

\begin{proof}
According to Theorem \ref{thm1}, we only need to find $g_{3k}(b), g_{2k}(b), g_{2k}(b), g_{0}(b)$ respectively. It follows from lemma \ref{lemma3} and lemma \ref{lemma4} that
\begin{align*}
g_{3k}(b)
&=\sum_{\substack{j_1+\cdots+j_5=k-3,\\ 0\leq j_i\leq k-1}} b^{j_3+j_4 +j_5}
=\sum_{l=0}^{k-3}\sum_{\substack{j_1+j_2=l,\\ 0\leq j_i\leq k-1}}\sum_{\substack{j_3+j_4+j_5=k-3-l,\\ 0\leq j_i\leq k-1}} b^{k-3-l}\\
&=\sum_{l=0}^{k-3}f(l)h(k-3-l) b^{k-3-l}
=\frac12\sum_{l=0}^{k-3}(l+1)(k-l-1)(k-l-2) b^{k-3-l}\\
&=\frac12\sum_{m=0}^{k-3}(k-2-m)(m+2)(m+1) b^{m},
\end{align*}

\begin{align*}
g_{2k}(b)
=&\sum_{l=0}^{k-2}\sum_{\substack{j_1+j_2=l,\\ 0\leq j_i\leq k-1}}\sum_{\substack{j_3+j_4+j_5=2k-3-l,\\ 0\leq j_i\leq k-1}} b^{2k-3-l}+
\sum_{l=k-1}^{2k-3}\sum_{\substack{j_1+j_2=l,\\ 0\leq j_i\leq k-1}}\sum_{\substack{j_3+j_4+j_5=2k-3-l,\\ 0\leq j_i\leq k-1}} b^{2k-3-l}\\
=&\frac12\sum_{l=0}^{k-2}(l+1)((2k-l-1)(2k-l-2)-3(k-l-1)(k-l-2)) b^{2k-3-l}\\
  &+\frac12\sum_{l=k-1}^{2k-3}(2k-l-1)^2(2k-l-2) b^{2k-3-l}\\
=&\frac12\sum_{m=0}^{k-2}((k-m-1)((k+m+1)(k+m)-3(m+1)m)b^{k-1}+(m+2)^2(m-1)) b^{m},
\end{align*}

\begin{align*}
g_{k}(b)
=&\sum_{l=0}^{k-2}\sum_{\substack{j_1+j_2=l,\\ 0\leq j_i\leq k-1}}\sum_{\substack{j_3+j_4+j_5=3k-3-l,\\ 0\leq j_i\leq k-1}} b^{3k-3-l}+
\sum_{l=k-1}^{2k-2}\sum_{\substack{j_1+j_2=l,\\ 0\leq j_i\leq k-1}}\sum_{\substack{j_3+j_4+j_5=3k-3-l,\\ 0\leq j_i\leq k-1}} b^{3k-3-l}\\
&+\sum_{l=2k-1}^{3k-3}\sum_{\substack{j_1+j_2=l,\\ 0\leq j_i\leq k-1}}\sum_{\substack{j_3+j_4+j_5=3k-3-l,\\ 0\leq j_i\leq k-1}} b^{3k-3-l}\\
=&\frac12\sum_{l=0}^{k-2}(l+1)^2(l+2) b^{3k-3-l}\\
 &+\frac12\sum_{l=k-1}^{2k-2}(2k-l-1)((3k-l-1)(3k-l-2)-3(2k-l-1)(2k-l-2)) b^{3k-3-l}\\
=&\frac12\sum_{m=0}^{k-1}((k-m-1)^2(k-m)b^{2k-1}+(m+1)((k+m+1)(k+m)-3(m+1)m)b^{k-1}) b^{m},
\end{align*}

\begin{align*}
g_{0}(b)
=&\sum_{l=0}^{k-1}\sum_{\substack{j_1+j_2=l,\\ 0\leq j_i\leq k-1}}\sum_{\substack{j_3+j_4+j_5=4k-3-l,\\ 0\leq j_i\leq k-1}} b^{4k-3-l}+
\sum_{l=k}^{2k-2}\sum_{\substack{j_1+j_2=l,\\ 0\leq j_i\leq k-1}}\sum_{\substack{j_3+j_4+j_5=4k-3-l,\\ 0\leq j_i\leq k-1}} b^{4k-3-l}\\
&+\sum_{l=2k-1}^{3k-3}\sum_{\substack{j_1+j_2=l,\\ 0\leq j_i\leq k-1}}\sum_{\substack{j_3+j_4+j_5=4k-3-l,\\ 0\leq j_i\leq k-1}} b^{4k-3-l}+
\sum_{l=3k-2}^{4k-3}\sum_{\substack{j_1+j_2=l,\\ 0\leq j_i\leq k-1}}\sum_{\substack{j_3+j_4+j_5=4k-3-l,\\ 0\leq j_i\leq k-1}} b^{4k-3-l}\\
=&\frac12\sum_{l=k}^{2k-2}(2k-l-1)(l-k+2)(l-k+1) b^{4k-3-l}\\
=&\frac12\sum_{m=0}^{k-2}(m+1)(k-m)(k-m-1) b^{2k-1+m}.
\end{align*}
This completes the proof.
\end{proof}

\begin{lemma}{\rm\cite{LD2}}\label{LD}
Let $\Omega$ be a bounded domain and  $p > 1$. If $\mathbf{P}$ maps $L^p(\Omega)$ to $A^p(\Omega)$
boundedly, then it also maps $L^q(\Omega)$ to $A^q(\Omega)$ boundedly, where $\frac1p+\frac1q=1$.
\end{lemma}

\begin{lemma}
	\label{necessity}
	Let $p \geq 1$ be any number outside of the interval $\left(\frac{2 k+2n}{k+n+1}, \frac{2 k+2n}{k+n-1}\right) .$ Then the Bergman projection $\mathbf{P}$ is not a bounded operator on $L^{p}\left(\Omega^{n+1}_k\right)$.
\end{lemma}

\begin{proof}
From (\ref{Bk}), we can write the Bergman kernel as
$$
B_{k,n}((z,w),(s,t))
 =\sum\limits_{p_1,q\in \Lambda_{k,n}}\frac{(z\cdot\overline{s})^{p_1}(w\cdot\overline{t})^{q}}{N_{k,n}(p_1,q)},
$$
where $N_{k,n}(p_1,q)$ is a normalizing constant. Now set $f(z,w):=z_1^m\bar{w}^l,$ where $(m,-l)\in \Lambda_{k,n}$ and $m+k(-l+1)=-n+1$.  This is a bounded function on $\Omega^{n+1}_{k}$, so $f \in L^{p}\left(\Omega^{n+1}_{k}\right)$ for all $p \geq 1$. It follows that
\begin{align*}
\mathbf{P}(f)(z,w) &=\int_{\Omega^{n+1}_{k}} B_{k,n}((z, w),(s,t)) f(s,t) d V(s,t)
=\int_{\Omega^{n+1}_{k}} \sum\limits_{p_1,q\in \Lambda_k}\frac{(z\cdot\overline{s})^{p_1}(w\cdot\overline{t})^{q}}{N_{k,n}(p_1,q)} s_1^m\bar{t}^l d V(s,t) \\
&=\sum\limits_{p_1,q\in \Lambda_{k,n}} \frac{1}{N_{k,n}(p_1,q)} \int_{|s|<1} \int_{|s|^k}^{1} \int_{0}^{2 \pi} (z\cdot\overline{s})^{p_1}s_1^mw^q(re^{-i\theta})^{q+l}r d\theta dr d V(s) \\
&=\sum\limits_{p_1\in \mathbb{N}} \frac{\pi}{N_{k,n}(p_1,-l)w^l} \int_{|s|<1} (z\cdot\overline{s})^{p_1}s_1^m(1-|s|^{2k}) d V(s) \\
&=\sum\limits_{p_1\in \mathbb{N}} \frac{\pi}{N_{k,n}(p_1,-l)w^l}\int_{0}^1r^{p_1+m+2n-1}(1-r^{2k})dr \int_{|\xi|=1}\frac{1}{2\pi}\int_{-\pi}^{\pi} (z\cdot\overline{\xi})^{p_1}\xi_1^m e^{(m-p_1)i\theta} d\theta d V(\xi) \\
&= \frac{\pi}{N_{k,n}(m,-l)w^l} \int_{|s|<1} (z\cdot\overline{s})^{m}s_1^m(1-|s|^{2k}) d V(s).
\end{align*}
Notice that
\begin{align*}
&\int_{|s|<1} (z\cdot\overline{s})^{m}s_1^m(1-|s|^{2k}) d V(s)\\
=& \int_{|s|<1} (z_1\overline{s_1}+z'\cdot\overline{s'})^{m}s_1^m(1-|s|^{2k}) d V(s)\\
=&\int_{|s'|<1}\int_{0}^{(1-|s'|^2)^{\frac{1}{2}}}r_1\int_{0}^{2\pi}\sum\limits_{{k_1}=0}^{m}(z_1r_1e^{-i\theta})^{{k_1}}(z'\cdot\overline{s'})^{m-{k_1}}(r_1e^{i\theta})^m d\theta (1-(r_1^2+|s'|^2)^{k})dr_1d V(s')\\
=&\int_{|s|<1} z_{1}^{m}\left|s_{1}\right|^{2 m}\left(1-|s|^{2 k}\right) d V(s)
=\frac{m!\pi^{n} B(\frac{m+n}{k},2)z_{1}^{m}}{k\Gamma(n+m)}.
\end{align*}
Therefore, $\mathbf{P}(f)(z,w)=\frac{C z_{1}^{m}}{w^{l}},$
where  $C=\frac{\pi^{n+1}m! B(\frac{m+n}{k},2)}{kN_{k,n}(m,-l)\Gamma(n+m)}$  is a constant. Thus,
$$
\begin{aligned}
\left\|\mathbf{P}(f)\right\|_{p}^{p}
&= C^p\int_{\Omega^{n+1}_{k}} \frac{| z_1|^{mp}}{\left|w\right|^{lp}} dV(z,w)\\
&=C^p\int_{0<|w|<1} |w|^{-lp}\int_{0}^{|w|^{\frac{1}{k}}}\int_{|\xi|=1}r^{mp+2n-1}| \xi_1|^{mp} drdV(\xi)dV(w) \\
&=\frac{C^p}{2n+mp}\int_{|\xi|=1}|\xi_1|^{mp} dV(\xi)\int_{0<|w|<1}|w|^{-lp+\frac{2n+mp}{k}}dV(w)\\
&=\frac{2\pi C^p}{2n+mp}\int_{|\xi|=1}|\xi_1|^{mp} dV(\xi)\int_{0}^1r^{-lp+\frac{2n+mp}{k}+1}dr.
\end{aligned}
$$
This integral diverges when $p \geq \frac{2 k+2n}{kl-m}=\frac{2 k+2n}{k+n-1}$, so $\mathbf{P}(f) \notin L^{p}\left(\Omega_{k}\right)$ for this range of $p$. Lemma \ref{LD} says that $\mathbf{P}$ also fails to be a bounded operator on $L^{p}\left(\Omega^{n+1}_{k}\right)$ when $p \in\left(1, \frac{2 k+2n}{k+n+1}\right]$. This completes the proof.
\end{proof}

\begin{lemma}(Schur's Lemma {\rm\cite{KZ}})\label{Schur}
 Let $\Omega \subset \mathbb{C}^{n}$ be a domain, $K$ be an a.e. positive, measurable function on $\Omega \times \Omega$, and $\mathcal{K}$ be the integral operator with kernel $K .$ Suppose there exists a positive auxiliary function $h$ on $\Omega$, and numbers $0<a<b$ such that for all $\epsilon \in[a, b)$, the following estimates hold:
$$
\begin{aligned}
\mathcal{K}\left(h^{-\epsilon}\right)(z) &:=\int_{\Omega} K(z, w) h(w)^{-\epsilon} d V(w) \lesssim h(z)^{-\epsilon} \\
\mathcal{K}\left(h^{-\epsilon}\right)(w) &:=\int_{\Omega} K(z, w) h(z)^{-\epsilon} d V(z) \lesssim h(w)^{-\epsilon} .
\end{aligned}
$$
Then $\mathcal{K}$ is a bounded operator on $L^{p}(\Omega)$, for all $p \in\left(\frac{a+b}{b}, \frac{a+b}{a}\right)$.
\end{lemma}

\begin{lemma}{\rm\cite{LD2}}
\label{estimate1}
 Let $D \subset \mathbb{C}$ be the unit disk, $\epsilon \in(0,1)$ and $\beta \in(-\infty, 2)$. Then for $z \in D$,
$$
\mathcal{I}_{\epsilon, \beta}(z):=\int_{D} \frac{\left(1-|w|^{2}\right)^{-\epsilon}}{|1-z \bar{w}|^{2}}|w|^{-\beta} \mathrm{d} V(w) \lesssim\left(1-|z|^{2}\right)^{-\epsilon}
$$
with constant independent of $z$.
\end{lemma}

\begin{lemma}\label{estimate2}
 Let $D_n \subset \mathbb{C}^n$ be the unit disk, $k\in \mathbb{Z}^+$, $\epsilon \in(0,1)$ and $\Delta \in \mathbb{C}^n, |\Delta|<1 $. Then
\begin{align}\label{lem9}
\int_{D_n} \frac{\left(1-|\eta|^{2k}\right)^{-\epsilon}}{|1-(\eta \cdot \overline{\Delta})^k|^{n+1}} \mathrm{d} V(\eta) \approx \left(1-|\Delta|^{2k}\right)^{-\epsilon}.
\end{align}
\end{lemma}

\begin{proof}
Choosing a unitary matrices $U$  such that $\Delta=|\Delta|\mathbbm{1}U^{-1},$ where $\mathbbm{1}=(1,0,\cdots,0)$.  Then we get
\begin{align}\label{lem9.1}
&\int_{D_n} \frac{\left(1-|\eta|^{2k}\right)^{-\epsilon}}{|1-(\eta \cdot \overline{\Delta})^k|^{n+1}} \mathrm{d} V(\eta)\notag\\
&=\int_{D_n} \frac{\left(1-|\eta|^{2k}\right)^{-\epsilon}}{|1-(|\Delta|\eta U (\mathbbm{1})^T )^k|^{n+1}} \mathrm{d} V(\eta)
=\int_{D_n} \frac{\left(1-|\eta|^{2k}\right)^{-\epsilon}}{|1-(|\Delta|\eta_1)^k|^{n+1}} \mathrm{d} V(\eta)\notag\\
&=\int_{D_n} \frac{\left(1-|\eta|^{2k}\right)^{-\epsilon}}{(1-(|\Delta|\eta_1)^k)^{\frac{n+1}2}(1-(|\Delta|\bar{\eta}_1)^k)^{\frac{n+1}2}} \mathrm{d} V(\eta)\notag\\
&=\sum_{m_1=0}^{\infty}\sum_{m_2=0}^{\infty}\int_{D_n} \left(1-|\eta|^{2k}\right)^{-\epsilon}\frac{\Gamma(m_1+\frac{n+1}2)}{\Gamma(m_1+1)\Gamma(\frac{n+1}2)}
\frac{\Gamma(m_2+\frac{n+1}2)}{\Gamma(m_2+1)\Gamma(\frac{n+1}2)}(|\Delta|\eta_1)^{km_1}(|\Delta|\bar{\eta}_1)^{km_2}) \mathrm{d} V(\eta)\notag\\
&=\sum_{m=0}^{\infty}\left(\frac{\Gamma(m+\frac{n+1}2)}{\Gamma(m+1)\Gamma(\frac{n+1}2)}\right)^2|\Delta|^{2km}
\int_{D_n} \left(1-|\eta|^{2k}\right)^{-\epsilon}|\eta_1|^{2km} \mathrm{d} V(\eta).
\end{align}
 From Lemma \ref{formula}, we have
\begin{align*}
\int_{D_n} \left(1-|\eta|^{2k}\right)^{-\epsilon}|\eta_1|^{2km} \mathrm{d} V(\eta)
&=\int_{0}^{1}(1-\rho^{2k})^{-\epsilon}\rho^{2n-1+2km}d\rho\int_{|\xi_1|=1}|\xi_1|^{2km}\mathrm{d}V(\xi)\\
&=\int_{0}^{1}(1-r)^{-\epsilon}r^{\frac{2n-1+2km}{2k}+\frac{1}{2k}-1}\frac{dr}{2k}\int_{|\xi_1|=1}|\xi_1|^{2km}\mathrm{d}V(\xi)\\
&=B(\frac{n+km}{k},1-\epsilon)\frac{\Gamma(km+1)\pi^n}{k\Gamma(n+km)}.
\end{align*}
Then
\begin{align*}
(\ref{lem9.1})
&=\frac{\Gamma(1-\epsilon)\pi^n}{k(\Gamma(\frac{n+1}2))^2}\sum_{m=0}^{\infty}(\frac{\Gamma(m+\frac{n+1}2)}{\Gamma(m+1)})^2
\frac{\Gamma(\frac{n}{k}+m)}{\Gamma(\frac{n}{k}+m+1-\epsilon)}\cdot \frac{\Gamma(km+1)}{\Gamma(km+n)}|\Delta|^{2km}.
\end{align*}
Considering that $\frac{\Gamma(m+\lambda)}{\Gamma(m)}\approx m^\lambda (m\rightarrow\infty)$ by Stirling's formula for $\lambda>0$, then we see that the coefficients in this last series are
of order $m^{\epsilon-1}$, as $m\rightarrow\infty$. This proves the assertions made about (\ref{lem9}).

\end{proof}
\subsection{Proof of Theorem 3}
Combined with Lemma \ref{necessity}, we prove Theorem 3, and only need to prove the following lemma.
\begin{lemma}
	\label{sufficiency}
	If $p \in (\frac{2k+2n}{k+n+1},\frac{2k+2n}{k+n-1})$, then the Bergman projection $\mathbf{P}$ maps $L^p(\Omega^{n+1}_{k})$ to $L^p(\Omega^{n+1}_k)$ boundedly.
\end{lemma}
\begin{proof}
From theorem \ref{thm1}, we see that
$$
B_{k,n}((z, w),(s,t))
=n!\frac{g_{(n+1) k}(b) a^{n+1}+\cdots+g_{k}(b) a+g_{0}(b)}{\pi^{n+1} k^{2}\left(1-a\right)^{2}\left(a-b^{k}\right)^{n+1}},
$$
where $a=w\cdot\overline{t}, b=z\cdot\overline{s},$ $g_{lk}(b) (0\leq l\leq n+1)$ is the polynomial
$$
g_{lk}(b)
=\sum_{\substack{j_1+j_2+\cdots+j_{n+3}=(n-l+2)k-(n+1),\\ 0\leq j_i\leq k-1}} b^{j_3+j_4+\cdots +j_{n+3}}.
$$
When $0\leq l\leq n-1$,
$$
|g_{lk}(b)|\lesssim b^{(n-l+2)k-(n+1)-2(k-1)}=b^{(n-l)k-(n-1)}.
$$
Since $\Omega^{n+1}_k$ is a bounded domain where $|s|^k < |t| < 1$, it follows that ${B_{k,n}}$ satisfies the crucial estimate
\begin{align}\label{th4.1}
|B_{k,n}((z, w),(s,t))|
\lesssim \frac{|a|^{n-\frac{n-1}{k}}}{|1-a|^{2}|a-b^{k}|^{n+1}}
=\frac{|w\cdot\overline{t}|^{n-\frac{n-1}{k}}}{|1-w\cdot\overline{t}|^{2}|w\cdot\overline{t}-(z\cdot\overline{s})^{k}|^{n+1}}.
\end{align}
Let $h(z,w):= (|w|^2-|z|^{2k})(1-|w|^2)$, we're going to prove that for all $\epsilon \in [\frac{k+n-1}{2k},\frac{k+n+1}{2k})$, and any $z \in \Omega^{n+1}_k$,
\begin{align}\label{th4.2}
|\mathbf{P}|\left(h^{-\epsilon}\right)(z,w) :=\int_{\Omega^{n+1}_k} |B_{k,n}((z, w),(s,t))| h(s,t)^{-\epsilon} d V(s,t) \lesssim h(z,w)^{-\epsilon}.
\end{align}
From estimate (\ref{th4.1}), we see that
\begin{align*}
|\mathbf{P}|\left(h^{-\epsilon}\right)(z,w)
&\lesssim  \int_{0<|t|<1}\int_{|s|^k<|t|}\frac{|w\cdot\overline{t}|^{n-\frac{n-1}{k}}(|t|^2-|s|^{2k})^{-\epsilon}(1-|t|^2)^{-\epsilon}}{|1-w\cdot\overline{t}|^{2}|w\cdot\overline{t}-(z\cdot\overline{s})^{k}|^{n+1}} dV(t)dV(s).
\end{align*}
Let
\begin{align*}
t=\rho e^{i\varphi}, s_m=r_m e^{i\theta_m}, z_m=|z_m| e^{i\tilde{\theta}_m}(m=1,\cdots,n),
w=|w| e^{i\varphi_0}, r=(r_1,\cdots,r_n),
\end{align*}
then
\begin{small}
\begin{align*}
&|\mathbf{P}|\left(h^{-\epsilon}\right)(z,w)\\
\lesssim & \int_{0}^1\int_0^{\rho^{\frac 1k}}(\prod_{m=1}^{n+1}\int_{-\pi}^{\pi})
\frac{|w|^{n-\frac{n-1}{k}}\rho^{n+1-\frac{n-1}{k}}(\rho^2-|r|^{2k})^{-\epsilon}(1-\rho^2)^{-\epsilon}r_1\cdots r_nd\varphi d\theta_1\cdots d\theta_ndrd\rho}
{|1-|w|\rho e^{i(\varphi_0-\varphi)}|^{2}||w|\rho e^{i(\varphi_0-\varphi)}-(|z_1|r_1 e^{i(\tilde{\theta}_1-\theta_1)}+\cdots+|z_n|r_n e^{i(\tilde{\theta}_n-\theta_n)})^{k}|^{n+1}}\\
=&\int_{0}^1\int_0^{\rho^{\frac 1k}}(\prod_{m=1}^{n+1}\int_{-\pi}^{\pi})
\frac{|w|^{n-\frac{n-1}{k}}\rho^{n+1-\frac{n-1}{k}}(\rho^2-|r|^{2k})^{-\epsilon}(1-\rho^2)^{-\epsilon}r_1\cdots r_nd\varphi d\theta_1\cdots d\theta_ndrd\rho}
{|1-|w|\rho e^{i\varphi}|^{2}||w|\rho-(|z_1|r_1 e^{i\theta_1}+\cdots+|z_n|r_n e^{i\theta_n)})^{k}|^{n+1}}.
\end{align*}
\end{small}
The last equation we used the periodicity of $\theta_1,\cdots, \theta_n$ and $\varphi$ integrals. Next, we first consider the following integral
\begin{align}\label{th4.3}
A:&=\int_0^{\rho^{\frac 1k}}(\prod_{m=1}^{n}\int_{-\pi}^{\pi})
\frac{(\rho^2-|r|^{2k})^{-\epsilon}r_1\cdots r_n d\theta_1\cdots d\theta_ndr}
{||w|\rho-(|z_1|r_1 e^{i\theta_1}+\cdots+|z_n|r_n e^{i\theta_n)})^{k}|^{n+1}}\notag\\
&=\int_0^{\rho^{\frac 1k}}(\prod_{m=1}^{n}\int_{-\pi}^{\pi})
\frac{\rho^{-2\epsilon}(1-|r\rho^{-1/k}|^{2k})^{-\epsilon}r_1\cdots r_n d\theta_1\cdots d\theta_ndr}
{(|w|\rho)^{n+1}|1 -(|z_1|r_1 (|w|\rho)^{-1/k} e^{i\theta_1}+\cdots+|z_n|r_n (|w|\rho)^{-1/k} e^{i\theta_n})^{k}|^{n+1}}.
\end{align}
Make the substitution $\tilde{r}_m=\frac{r_m}{\rho^{1/k}} (m=1,\cdots,n),$ and let
$$\eta=(\tilde{r}_1e^{i\theta_1},\cdots, \tilde{r}_ne^{i\theta_n}),
\Delta=(|z_1||w|^{-1/k},\cdots, |z_n||w|^{-1/k}),$$ then
\begin{align*}
(\ref{th4.3})
&=\int_0^{1}(\prod_{m=1}^{n}\int_{-\pi}^{\pi})
\frac{\rho^{-2\epsilon-(n+1)+\frac{2n}k}(1-|\tilde{r}|^{2k})^{-\epsilon}\tilde{r}_1\cdots\tilde{r_n} d\theta_1\cdots d\theta_nd\tilde{r}}
{|w|^{n+1}|1 -(|z_1| |w|^{-1/k} \tilde{r}_1 e^{i\theta_1}+\cdots+|z_n| |w|^{-1/k} \tilde{r}_n e^{i\theta_n})^{k}|^{n+1}}\\
&=\frac{\rho^{-2\epsilon-(n+1)+\frac{2n}k}}{|w|^{n+1}}\int_{D_n} \frac{\left(1-|\eta|^{2k}\right)^{-\epsilon}}{|1-(\eta \cdot \Delta)^k|^{n+1}} \mathrm{d} V(\eta).
\end{align*}
From lemma \ref{estimate2}, we have
\begin{align*}
A\approx \frac{\rho^{-2\epsilon-(n+1)+\frac{2n}k}}{|w|^{n+1}}\left(1-|\Delta|^{2k}\right)^{-\epsilon}
=\rho^{-2\epsilon-(n+1)+\frac{2n}k}|w|^{2\epsilon-(n+1)}\left(|w|^{2}-|z|^{2k}\right)^{-\epsilon}.
\end{align*}
This means that
\begin{align*}
&|\mathbf{P}|\left(h^{-\epsilon}\right)(z,w)\\
\lesssim& |w|^{2\epsilon -\frac{n-1}{k}-1}\left(|w|^{2}-|z|^{2k}\right)^{-\epsilon}
\int_{0}^1\int_{-\pi}^{\pi}
\frac{\rho^{\frac{n+1}{k}-2\epsilon}(1-\rho^2)^{-\epsilon}d\varphi d\rho}
{|1-|w|\rho e^{i\varphi}|^{2}}\\
=&|w|^{2\epsilon -\frac{n-1}{k}-1}\left(|w|^{2}-|z|^{2k}\right)^{-\epsilon}
\int_{D}
\frac{|z|^{\frac{n+1}{k}-2\epsilon-1}(1-|z|^2)^{-\epsilon}dz}
{\left|1-|w|z\right|^{2}}.
\end{align*}
From lemma \ref{estimate1},
\begin{align*}
|\mathbf{P}|\left(h^{-\epsilon}\right)(z,w)
\lesssim |w|^{2\epsilon -\frac{n-1}{k}-1}\left(|w|^{2}-|z|^{2k}\right)^{-\epsilon}\left(1-|w|^{2}\right)^{-\epsilon},
\end{align*}
when $\frac{n+1}{k}-2\epsilon-1 > -2, i.e. \frac{k+n+1}{2k}>\epsilon$. Then,
\begin{align*}
|\mathbf{P}|\left(h^{-\epsilon}\right)(z,w)
\lesssim\left(|w|^{2}-|z|^{2k}\right)^{-\epsilon}\left(1-|w|^{2}\right)^{-\epsilon}=h(z,w)^{-\epsilon},
\end{align*}
when $2\epsilon -\frac{n-1}{k}-1 \geq0, i.e. \frac{k+n-1}{2k}\leq \epsilon$.  This completes the proof of (\ref{th4.2}). Finally, combining (\ref{th4.2}) and Schur's lemma (Lemma \ref{Schur}) yields that the operator $|\mathbf{P}|$ is bounded from $L^p(\Omega^{n+1}_{k})$ to $L^p(\Omega^{n+1}_{k})$ for $p \in (\frac{2k+2n}{k+n+1},\frac{2k+2n}{k+n-1})$.
A fortiori, $\mathbf{P}$ is bounded from $L^p(\Omega^{n+1}_{k})$ to $L^p(\Omega_{k})$ for $p$ in the same range. Notice that because of the conjugate symmetry
of $B_{k,n}$, it is sufficient to establish just one of the estimates to apply Lemma \ref{Schur}. This completes the proof.

\end{proof}

\section*{Acknowledgments}
Special thanks to referees for many  constructive comments and suggestions, which are valuable and very helpful for improving our manuscript. The project is supported by National Natural Science Foundation of China (Grant no. 12071035 and Grant no. 11971045).

\vskip 0.5cm{
\end{document}
\parindent=0pt
\begin{thebibliography}{99}

\bibitem{LD1} L. D. Edholm, {\it The Bergman kernel of fat Hartogs triangles ,}  Doctoral dissertation, Ohio State University  (2016).
\bibitem{LD2} L. D. Edholm, J. D. McNeal,  {\it The Bergman projection on fat Hartogs triangles: $L^p$ boundedness}, Proc. Amer. Math. Soc. \textbf{144}, 2185-2196 (2016).
\bibitem{LD3} L. D. Edholm, {\it Bergman theory of certain generalized Hartogs triangles}, Pacific J. Math. \textbf{284}(2), 327-342 (2016).
\bibitem{LD4} L. D. Edholm, J. D. McNeal,
{\it Sobolev mapping of some holomorphic projections,} J. Geom. Anal. \textbf{30}(2), 1293-1311 (2020).
\bibitem{DG1} G. T. Deng, {\it Complex analysis (in Chinese)}, Beijing Normal Universty Press (2010).
\bibitem{BELL} S. R. Bell, {\it  The Bergman kernel function and proper holomorphic mappings,}  Trans. Amer. Math. Soc. \textbf{270}(2), 685-691  (1982).
\bibitem{HB} H. Bommier-Hatoa, M. Englis, E. H. Youssfi, {\it Bergman-type projections in generalized Fock spaces,}  J. Math. Anal. Appl. \textbf{389}(2), 1086-1104 (2012).
\bibitem{WR} W. Rudin, {\it Function theory in the unit ball in $\mathbb{C}^n$} , vol. 241 of Grundlehren der Mathematischen Wissenschaften. Springer-Verlag, New York (1980).
\bibitem{TB}   T. Beberok, {\it $L^p$ boundedness of the Bergman projection on some generalized
Hartogs triangles}, Bull. Iran.Math. Soc. \textbf{43}, 2275-2280 (2017).
\bibitem{ZH} Z. Huo,  {\it $L^p$ estimates for the Bergman projection on some Reinhardt domains},
Proc. Am. Math. Soc. \textbf{146}, 2541-2553 (2018).
\bibitem{ZB} Z Blocki, {\it The Bergman metric and the pluricomplex Green function}, Trans.
Am. Math. Soc. \textbf{357}, 2613-2625 (2005).
\bibitem{DG2} G. T. Deng, Y. Huang,  T. Qian, {\it Reproducing Kernels of Some Weighted Bergman Spaces}, J. Geom. Anal. \textbf{31}, 9527-9550 (2021).
\bibitem{CL1} L. Chen, {\it The $L^p$ boundedness of the Bergman projection for a class of bounded
Hartogs domains}, J. Math. Anal. Appl. \textbf{448}, 598-610 (2017).
\bibitem{CL2} L. Chen, {\it Weighted Sobolev regularity of the Bergman projection on the Hartogs triangle.} Pac. J. Math. \textbf{288}, 307-318 (2017).
\bibitem{KZ} K. Zhu, {\it Spaces of holomorphic functions in the unit ball,} Graduate Texts in Mathematics 266 Springer, New York, NY  (2005).
\bibitem{CD} D. Chakrabarti, A. Konkel, M. Mainkar, E. Miller, {\it Bergman kernels of elementary Reinhardt domains,} Pacific J. Math. \textbf{306}(1), 67-93 (2020). 



\end{thebibliography}
